\documentclass{amsart}
\usepackage{amssymb}
\usepackage{amscd}
\theoremstyle{plain}
\newtheorem{theorem}{Theorem}[section]
\newtheorem{proposition}[theorem]{Proposition}
\newtheorem{corollary}[theorem]{Corollary}
\newtheorem{lemma}[theorem]{Lemma}

\theoremstyle{definition}
\newtheorem{definition}[theorem]{Definition}
\newtheorem{example}[theorem]{Example}
\newtheorem{remark}[theorem]{Remark}

\newtheorem{conjecture}[theorem]{Conjecture}

\theoremstyle{remark}
\newtheorem{notation}[theorem]{Notation}

\numberwithin{equation}{section}

\newcommand{\N}{\mathbb N}
\newcommand{\Z}{\mathbb Z}
\newcommand{\Q}{\mathbb Q}
\newcommand{\R}{\mathbb R}
\newcommand{\C}{\mathbb C}

\newcommand{\Id}{\mathrm{Id}}

\newcommand{\Ot}{\operatorname{O}}

\newcommand{\SO}{\operatorname{SO}}
\newcommand{\spec}{\operatorname{Spec}}

\newcommand{\Spin}{\operatorname{Spin}}
\newcommand{\Pin}{\operatorname{Pin}}
\newcommand{\Le}{L}

\newcommand{\diag}{\operatorname{diag}}
\newcommand{\mult}{\operatorname{mult}}

\newcommand{\bs}{\backslash}

\begin{document}

\title{Properties of the Dirac spectrum on three dimensional lens spaces}
\author{Sebastian Boldt}
\address{Institut f\"ur Mathematik, Humboldt-Universit\"at zu Berlin, D-10099 Berlin, Germany}
\email{boldt@math.hu-berlin.de}
\date{March 31st 2015}

\thanks{The author was supported by DFG Sonderforschungsbereich 647}

\maketitle
\begin{abstract}
	We present a spectral rigidity result for the Dirac operator on lens spaces. More specifically, we show that each homogeneous lens space and each three dimensional lens space $\Le(q;p)$ with $q$ prime is completely characterized by its Dirac spectrum in the class of all lens spaces.
\end{abstract}
\section{Introduction}

Let $M$ be a compact Riemannian spin manifold, i.e., a compact oriented Riemannian manifold with a fixed spin structure. Then there is a canonical first order differential operator $D$ on $M$ called the (spin-) Dirac or Atiyah-Singer operator. This operator is elliptic and self-adjoint and hence possesses a discrete real spectrum $\spec_D(M)$ consisting of eigenvalues with finite multiplicities. A typical question in spectral geometry is to what extent the geometry of $M$ is determined by $\spec_D(M)$ or by the spectrum of any other canonical geometric operator.

In this article, we investigate this question for three dimensional lens spaces and for homogeneous lens spaces. A lens space is a quotient of the sphere $S^{2m-1}$ by a cyclic group of isometries.

Spectrally, lens spaces were first examined by A.\ Ikeda and Y.\ Yamamoto in \cite{0415.58018} and \cite{0426.10028}. These authors proved that homogeneous lens spaces and three dimensional lens spaces are completely determined by the spectra of their Laplace-Beltrami operators in the class of all lens spaces.

Spectra are seldom explicitly computable. For a $2m-1$ dimensional lens space $\Le=\Le(q;p_1,\ldots,p_m)$ (see Section~\ref{sec:setup} for notation and definitions), the eigenvalues of the Laplacian are $k(k+2(m-1)),\; k\in\N_0$, with corresponding multiplicities $m_k$ which depend on the lens space $\Le$. Ikeda and Yamamoto introduced the generating function $F^\Le$ of the Laplace operator on $\Le$, defined by
$$
F^\Le(z) = \sum_{k=0}^{\infty}m_kz^k\,.
$$
By definition, this power series encodes the whole spectrum. Hence, two lens spaces $\Le$ and $\Le'$ are Laplace isospectral if and only if $F^\Le=F^{\Le'}$. Ikeda and Yamamoto prove that the generating functions have meromorphic extensions to $\C$ of the form
\begin{equation}\label{eqn:laplace-generating-function}
	F^\Le(z)=\frac{1}{q}\sum_{k=0}^{q-1}\frac{1-z^2}{\prod_{j=1}^m (\xi_q^{k p_j}-z)(\xi_q^{- k p_j}-z)}
\end{equation}
where $\xi_q$ denotes the $q$-th root of unity $e^{2\pi i/q}$.

They then use formula \eqref{eqn:laplace-generating-function} to prove spectral rigidity of homogeneous lens spaces by determining the order of the poles of $F^\Le$. In the non-homogeneous case, a careful analysis of the poles and residues of $F^\Le$ leads to the following set of equations if the three dimensional lens spaces $\Le(q;p)$ and $\Le(q;s)$ are isospectral:
\begin{align}\label{eqn:intro_0}
\cot\frac{k(p+1)}{q}\pi-\cot\frac{k(p-1)}{q}\pi+\cot\frac{k(p^*+1)}{q}\pi-\cot\frac{k(p^*-1)}{q}\pi\\
= \cot\frac{k(s+1)}{q}\pi-\cot\frac{k(s-1)}{q}\pi+\cot\frac{k(s^*+1)}{q}\pi-\cot\frac{k(s^*-1)}{q}\pi\notag
\end{align}
for all $1\le k\le q-1$ s.t. $k(p\pm1)\not\equiv 0\pmod q$ and $k(s\pm1)\not\equiv 0\pmod q$, where $p^*$ and $s^*$ are multiplicative inverses modulo $q$ of $p$ and $s$, respectively. The solutions of this set of equations are $p\equiv\pm s\pmod q$ and $p\cdot s\equiv\pm1\pmod q$. This, in turn, is equivalent to $\Le(q;p)$ and $\Le(q;s)$ being isometric.

The key ingredient to solving the equations in \eqref{eqn:intro_0} is the fact that the numbers $\cot\frac{k}{q}\pi$ with $1\le k \le \frac{q}{2}, (k,q)=1$ are linearly independent over $\Q$ (see e.g.\ \cite{0267.10065}), a fact that goes back to a very interesting problem of Chowla \cite{MR0249393}.  

Inspired by Ikeda's and Yamamoto's work, C. B\"ar introduced the generating functions of the Dirac operator on lens spaces, and more generally on spherical space forms, in his Ph.D. thesis \cite{0748.53022} (see also \cite{0848.58046}). The eigenvalues of the Dirac operator on the lens space $\Le=\Le(q;p_1,\ldots,p_m)$ with a fixed spin structure are $\pm\left(\frac{2m-1}{2}+k\right)$, $k\in\N_0$, with corresponding multiplicities $m_k^\pm$. The associated generating functions are
\begin{equation*}
F_\pm^\Le(z)=\sum\limits_{k=0}^{\infty}m_k^\pm z^k\,.
\end{equation*}
B\"ar proves that these generating functions have meromorphic extensions to $\C$. For odd $q$ these are (see Corollary~\ref{cor:generating_functions_lens_spaces}):
\begin{align*}
	\begin{split}
		F_+^\Le (z)= \frac{2^{m-1}}{q}\sum_{k=0}^{q-1}\frac{\displaystyle\sum_{\epsilon_1\cdots\epsilon_m=(-1)^{m+1}}\xi_{2q}^{(q+1)k\sum_j\epsilon_j p_j} -z\cdot \!\!\!\!\!\!\!\sum_{\epsilon_1\cdots\epsilon_m=(-1)^{m}}\xi_{2q}^{(q+1)k\sum_j\epsilon_j p_j}}{\displaystyle\prod_{j=1}^m (\xi_q^{k p_j}-z)(\xi_q^{- k p_j}-z)}\,,
	\end{split}\\
	\begin{split}
		F_-^{\Le} (z)= \frac{2^{m-1}}{q}\sum_{k=0}^{q-1}\frac{\displaystyle\sum_{\epsilon_1\cdots\epsilon_m=(-1)^{m}}\xi_{2q}^{(q+1)k\sum_j\epsilon_j p_j} -z\cdot \!\!\!\!\!\!\!\sum_{\epsilon_1\cdots\epsilon_m=(-1)^{m+1}}\xi_{2q}^{(q+1)k\sum_j\epsilon_j p_j}}{\displaystyle\prod_{j=1}^m (\xi_q^{k p_j}-z)(\xi_q^{- k p_j}-z)}\,.
	\end{split}
\end{align*}

The aim of this article is to shed light on the relationship between the Dirac spectrum and the geometry of lens spaces which are homogeneous or three-dimensional. We will use B\"ar's formulas and the path that has been paved by Ikeda and Yamamoto to do so. Namely, we will analyse the poles and residues of $F_\pm^\Le$, which encode spectral information, to deduce geometric properties of $\Le$.  

There are, however, some differences to the Laplace case. The Dirac spectrum depends not only on the metric, but also on the orientation and the spin structure. The dependence on the orientation is already immanent in the definition of a spin structure and one can easily see that a change of orientation causes the spectrum to be reflected about zero. The second dependence is more complicated, i.e., there is in general no relation between the spectra associated to inequivalent spin structures; see, e.g., \cite{0995.58020}. If, however, the Riemannian spin manifold $M$ has two inequivalent spin structures $(P,\xi)$ and $(Q,\eta)$ and an, say orientation preserving, isometry $f$ that \textit{relates} the two spin structures, i.e., the differential $\mathrm{d}f:\SO(M)\to\SO(M)$ as a map of the oriented orthonormal frame bundle $\SO(M)$ of $M$ lifts to a map $\widetilde{\mathrm{d}f}:P\to Q$, then the spectra of the Dirac operators associated with the spin structures $(P,\xi)$ and $(Q,\eta)$ coincide. 

Because of the above, we introduce the notion of two Riemannian spin manifolds being \textit{$\varepsilon$-spin-isometric}, which means there is an isometry that relates their spin structures and is, according to $\varepsilon=1$ or $\varepsilon=-1$, orientation preserving or reversing. We then extend the well-known Theorem~\ref{thm:lensspace-reg-isometry}, which states an equivalent condition of isometry for two lens spaces, to Theorems~\ref{thm:lensspace-or-isometry} and \ref{thm:lensspace-or-spin-isometry} which state equivalent conditions for $\varepsilon$-spin-isometry of two lens spaces.

Of course, $\varepsilon$-spin-isometric manifolds are $\varepsilon$-isospectral; that is for $\varepsilon=1$ their spectra coincide (including multiplicities) whereas for $\varepsilon=-1$ their spectra coincide after one of them is reflected about zero. Isospectrality is then understood as $\varepsilon$-isospectrality for some $\varepsilon\in\{\pm1\}$.

The two main results of this paper, Theorems~\ref{thm:homo-lensspace-rigidity} and \ref{thm:mainthm}, are concerned with the inverse direction. Theorem~\ref{thm:homo-lensspace-rigidity} states that if two lens spaces, one of which is homogeneous, are isospectral, then they are isometric. In dimension three, this statement is strengthened to state that two $\varepsilon$-isospectral lens spaces, one of which is homogeneous, are $\varepsilon$-spin-isometric.

Theorem~\ref{thm:mainthm} states that for prime $q$, if the lens spaces $\Le(q;p)$ and $\Le(q;s)$ are $\varepsilon$-isospectral, then they are $\varepsilon$-spin-isometric. The proof starts out for general $q$ in the same way as Ikeda's and Yamamoto's proof did and arrives at the following set of equations in the case that $q$ is odd (the case $q$ even is similar, see Corollary~\ref{cor:isospec-lensspaces-eqs}) and both lens spaces are non-homogeneous:
\begin{align}\label{eqn:intro_3}
	\xi_q^{\frac{q+1}2kp}\left(\cot\frac{k(p+1)}{q}\pi-\cot\frac{k(p-1)}{q}\pi\right)+\xi_q^{\frac{q+1}2kp^*}\left(\cot\frac{k(p^*+1)}{q}\pi-\cot\frac{k(p^*-1)}{q}\pi\right)\\
	= \xi_q^{\frac{q+1}2ks}\left(\cot\frac{k(s+1)}{q}\pi-\cot\frac{k(s-1)}{q}\pi\right)+\xi_q^{\frac{q+1}2ks^*}\left(\cot\frac{k(s^*+1)}{q}\pi-\cot\frac{k(s^*-1)}{q}\pi\right)\notag
\end{align}
for all $1\le k\le q-1$ s.t. $k(p\pm1)\not\equiv 0\pmod q$ and $k(s\pm1)\not\equiv 0\pmod q$.
Unlike in the case of the numbers $\cot\frac{k}{q}\pi$, $(k,q)=1$, there is no theorem about the linear independence of the numbers $\xi_q^k\cot\frac{l}{q}\pi$. In fact, it is hard to even formulate such a statement as there are nontrivial linear dependences among these numbers due to dimensional reasons.

However, at the end of \cite{0426.10028}, Yamamoto gave an alternative proof for the solutions of the equations \eqref{eqn:intro_0} in case $q$ is a prime number. That proof uses techniques from analytic number theory, namely, the theory of $\lambda$-adic series in the cyclotomic fields $\Q_q=\Q\left(\xi_q\right)$, and it carries over to the case of the Dirac operator. Thus, we will use this technique to solve the equations \eqref{eqn:intro_3}.

Unfortunately, that proof does not work for arbitrary $q\in\N$. Numerical calculations for $q$ in a large range suggest that Theorem~\ref{thm:mainthm} is true if one drops the assumption that $q$ is prime, see Remark~\ref{rem:main-conjecture_consequences} and Conjecture~\ref{con:all_q}, but the method of proof would have to be different.

This paper's main focus lies on three dimensional lens spaces; one can also ask about Dirac isospectrality of higher dimensional lens spaces. This has been investigated in \cite{2014arXiv1412.2599B}. With a representation theoretic approach, explicit formulas for the multiplicities $m_k^\pm$ have been found, and from this criteria for the isospectrality of lens spaces have been derived. Furthermore, we were able to construct families of isospectral lens spaces in higher dimensions.

We refer the reader who is new to spin geometry to \cite{0688.57001} for a comprehensive introduction. A thorough treatment of the Dirac spectrum is given in \cite{1186.58020}. In particular, several examples of Dirac isospectral pairs and families are given in that book. 

This paper is organised as follows. In Section~\ref{sec:setup} we introduce lens spaces, describe their spin structures and their $+1$-spin-isometry classes. Section~\ref{sec:spectra} contains the description of the spectra of lens spaces via generating functions and closes with Theorem~\ref{thm:homo-lensspace-rigidity}, the statement about the spectral rigidity of homogeneous lens spaces. In Section~\ref{sec:three-dim}, we restrict to three dimensional non-homogeneous lens spaces and carry out the details that lead to Theorem~\ref{thm:mainthm} and Conjecture~\ref{con:all_q}. In Section~\ref{sec:eta-invariant}, we mention how the technique used in Section~\ref{sec:three-dim} carries over to give a new proof a known result about $\eta$-invariants of three dimensional lens spaces.

\section*{Dedication}
I dedicate this article to J\"org Sch\"ulke, my dear father and first teacher in mathematics.
\section{Lens Sapces, Spin Structures and Isometry classes}\label{sec:setup}

This section contains the setup used throughout the paper. We define lens spaces, describe their spin structures and isometry classes.  

\begin{definition}

Let $q\in\N$, $p_1,\ldots,p_m\in\Z$ with $(q,p_i)=1$ for $1\leq i\leq m$. Define $\gamma_q^p$ by
$$
\gamma_q^p=\gamma_q^{p_1,\ldots,p_m} := \diag\left(\left[
\begin{smallmatrix}
 \cos(2p_1\pi/q)&\sin(2p_1\pi/q)\\
 -\sin(2p_1\pi/q)&\cos(2p_1\pi/q)
 \end{smallmatrix}\right],\ldots,\left[
\begin{smallmatrix}
 \cos(2p_m\pi/q)&\sin(2p_m\pi/q)\\
 -\sin(2p_m\pi/q)&\cos(2p_m\pi/q)
 \end{smallmatrix}
 \right]\right)\in\SO(2m)\,.
$$
The lens space $\Le(q;p_1,\ldots,p_m)$ is defined as the quotient
$$
\Le(q;p_1,\ldots,p_m) = \langle \gamma_q^{p_1,\ldots,p_m} \rangle\setminus S^{2m-1}\, .
$$
\end{definition}
The matrix $\gamma_q^{p_1,\ldots,p_m}$ is of order $q$ and has, by assumption, only primitive $q$-th roots of unity as eigenvalues. Therefore, it generates a freely acting finite group of orientation preserving isometries of $S^{2m-1}$, and $\Le(q;p_1,\ldots,p_m)$ is thus canonically given the structure of an oriented Riemannian manifold.

The bundle of oriented orthonormal frames $\SO(S^{2m-1})$ of $S^{2m-1}$ is $\SO(2m)$ with projection onto the last column vector. A spin structure is a fibrewise non-trivial two-sheeted covering of the bundle of oriented orthonormal frames (see \cite[Theorem 1.4]{0688.57001}), which, in this case, is seen to be $\Spin(2m)$. Due to its simply-connectedness, this is the only spin structure of $S^{2m-1}$. Since lens spaces $\Le(q;p_1,\ldots,p_m)$ are quotients of the sphere, their spin structures arise as certain quotients of the sphere's spin structure $\Spin(2m)$ (see \cite[Proposition 1.4.2]{1186.58020}), which are in one-to-one correspondence with group homomorphisms $\tau:\Gamma\to\Spin(2m)$ such that $\Theta\circ \tau=\Id_\Gamma$, where $\Theta:\Spin(2m)\to\SO(2m)$ is the universal covering homomorphism.

Spin structures on lens spaces were first classified in \cite{0645.57019}, though not in the language described above.

\begin{proposition}\label{prop:lensspace-spin-str}
The lens space $\Le=\Le(q;p_1,\ldots,p_m)$ admits a spin structure if and only if $q$ is odd or $m$ is even. If $q$ is odd, the unique spin structure is given by
$$
\tau\left(\left(\gamma_q^p\right)^k\right):=\prod_{j=1}^m \left(\cos \tfrac{k(q+1)p_j\pi}{q}+\sin\tfrac{k(q+1)p_j\pi}{q}e_{2j-1}e_{2j} \right)\, .
$$
If $q$ and $m$ are even, there are precisely two spin structures given by
$$
\tau_h\left(\left(\gamma_q^p\right)^k\right):=(-1)^{k(h+h_q^p)}\prod_{j=1}^m \left(\cos \tfrac{kp_j\pi}{q}+\sin\tfrac{kp_j\pi}{q}e_{2j-1}e_{2j} \right)\,,
$$
for $h\in\{0,1\}$, where $h_q^p:=h_q^{p_1,\ldots,p_m}:=\sum_{j=1}^{m}\lfloor \tfrac{p_j}q \rfloor$.
\end{proposition}
\begin{proof}
A group homormorphism $\tau:\langle\gamma_q^p\rangle\to\Spin(2m)$ with $\Theta\circ\tau=
\Id$ has to map $\gamma_q^p$ to one of its preimages under $\Theta$, which are $\pm\prod_{j=1}^m \left(\cos \frac{kp_j\pi}{q}+\sin\frac{kp_j\pi}{q}e_{2j-1}e_{2j} \right)$ (see, e.g.,\ \cite[pp. 173]{MR781344}). It thus suffices to determine the order of these elements. By elementary calculations in the group $\Spin(2m)$ (see, e.g.,\ \cite[Chapter I]{0688.57001})
\begin{equation*}
\begin{split}
\left(\pm\prod_{j=1}^m \left(\cos \frac{kp_j\pi}{q}+\sin\frac{kp_j\pi}{q}e_{2j-1}e_{2j} \right)\right)^q & = \left(\pm 1\right)^q\prod_{j=1}^m \left(\cos\left( kp_j\pi\right)+\sin\left(kp_j\pi\right) e_{2j-1}e_{2j} \right)\\
&=\left(\pm 1\right)^q \prod_{j=1}^{m}\left(-1\right)^{p_j}\\
&= \left(\pm 1\right)^q(-1)^{\sum_j p_j}\, .
\end{split}
\end{equation*}
If $q$ is odd, exactly one of the two preimages, depending on the parity of $\sum_j p_j$, has order $q$. If $q$ es even, the $p_i$ are necessarily odd so that each preimage has order $q$ if and only if $m$ is even. For $\tau_0$ and $\tau_1$ to be invariant modulo $q$ instead of $2q$, we introduce $h_q^p$.
\end{proof}
\begin{notation}
If $q$ is odd, let $\Le(q;p_1,\ldots,p_m)$ be equipped with its unique spin structure. If, on the other hand, $q$ and $m$ are even, denote by $\Le(q;p_1,\ldots,p_m;h)$ the lens space $\Le(q;p_1,\ldots,p_m)$ together with the spin structure $\tau_h$, where $h\in\{0,1\}$.
\end{notation}

We now turn our attention to the isometry classes of lens spaces. The following theorem is well known.
\begin{theorem}\label{thm:lensspace-reg-isometry}
Let $\Le:=\Le(q;p_1,\ldots,p_m)$ and $\Le':=\Le(q;s_1,\ldots,s_m)$. Then the following assertions are equivalent:
\begin{enumerate}
\item $\Le$ and $\Le'$ are isometric.
\item There is a number $\ell\in\Z$, a permutation $\sigma\in S_m$ and there are numbers $\varepsilon_i\in\{\pm 1\}$ such that
$$ \ell p_{\sigma(i)} \varepsilon_{\sigma(i)} \equiv s_i \pmod q $$
for every $1\leq i \leq m$.
\end{enumerate}
\end{theorem}
\begin{proof}
	$(1)\to(2)$: If $\Le$ and $\Le'$ are isometric, they are certainly homeomorphic. Now \cite[\S 31]{0261.57009} asserts (2).
	
	To prove $(2)\to(1)$, we associate the $S^{2m-1}$-isometry
	\begin{equation}\label{eq:explicit-isometry}
		\Psi(x_1,y_1,\ldots,x_m,y_m):=(x_{\sigma(1)},\varepsilon_{\sigma(1)}y_{\sigma(1)},\ldots,x_{\sigma(m)},\varepsilon_{\sigma(m)}y_{\sigma(m)})
	\end{equation}
	with the data $\ell, \sigma$ and $\varepsilon_i,\, 1\le i\le m$.
	One easily checks that
	\begin{equation}\label{eq:generator-relation}
		(\gamma_q^p)^\ell=\Psi^{-1}\circ\gamma_q^s\circ\Psi\,,
	\end{equation}
	so that $\Psi$ induces an isometry $\psi:\Le\to\Le'$ by $[x] \mapsto [\Psi(x)]$.
\end{proof}

For any odd dimensional oriented Riemannian manifold $M$ we denote by $\tilde{M}$ the reversely oriented Riemannian manifold $M$ and define a $\SO(n;\R)$-equivariant map $F:\SO(M)\to\SO(\tilde{M})$ by $(v_1,\ldots,v_n)_x\mapsto (-v_1,\ldots,-v_n)_x=(v_1,\ldots,v_n)_x\cdot (-\Id)$,  where $(v_1\ldots,v_n)$ is an orthonormal basis of $T_xM$. Next, we associate with any spin structure $(P,\xi)$ of $M$ a spin structure $(\tilde{P},\tilde{\xi})$ of $\tilde{M}$ by setting $\tilde{P}:=P$ and $\tilde{\xi}:=F\circ\xi:\tilde{P}\to\SO(M)$.

Let $M$ and $N$ be odd dimensional spin manifolds with spin structures $(P,\xi)$ and $(Q,\eta)$, respectively, and let $f:M\to N$ be a smooth map.

\begin{definition}\label{def:spin-isometry}
	\begin{enumerate}
		\item The map $f$ \textit{relates} the spin structures $(P,\xi)$ and $(Q,\eta)$ if $f$ is orientation preserving and $f^*(Q,\eta)$ is equivalent to $(P,\xi)$, or if $f$ is orientation reversing and $f^*(Q,\eta)$ is equivalent to $(\tilde{P},\tilde{\xi})$.
		\item The map $f$ is a $+1$-\textit{isometry} if it is an orientation preserving isometry and a $-1$-\textit{isometry} if it is an orientation reversing isometry.
		\item The spin manifolds $M$ and $N$ are \textit{spin-isometric} if there exists an isometry $f:M\to N$ that relates their spin structures.
		\item Let $\varepsilon\in\{\pm 1\}$. Then $M$ and $N$ are \textit{$\varepsilon$-spin-isometric} if there exists an $\varepsilon$-isometry $f:M\to N$ that relates their spin structures.
	\end{enumerate}
\end{definition}

\begin{theorem}\label{thm:lensspace-or-isometry}
	Let $\varepsilon\in\{\pm 1\}$ and $\Le:=\Le(q;p_1,\ldots,p_m)$, $\Le':=\Le(q;s_1,\ldots,s_m)$. Then the following assertions are equivalent:
	\begin{enumerate}
		\item $\Le$ and $\Le'$ are $\varepsilon$-isometric.
		\item There is a number $\ell\in\Z$, a permutation $\sigma\in S_m$ and there are numbers $\varepsilon_i\in\{\pm 1\}$ such that
		\begin{eqnarray*}
			l p_{\sigma(i)} \varepsilon_{\sigma(i)} &\equiv & s_i \pmod {q} \quad   \forall 1\le i\le m\\
			\prod_{i=1}^m \varepsilon_i & = & \varepsilon .
		\end{eqnarray*}
	\end{enumerate}
\end{theorem}
\begin{proof}
	$(1)\to(2)$: We assume $q>2$. Let $f:\Le\to\Le'$ be an $\varepsilon$-isometry. Identify $\pi_1(\Le)$ and $\pi_1(\Le')$ with $\langle \gamma_q^p \rangle$ and $\langle \gamma_q^s \rangle$, respectively, and denote by $f_\#:\pi_1(\Le)\to\pi_1(\Le')$ the induced map on fundamental groups. Define $\ell\in\Z$ as the smallest non-negative integer satisfying
	$$
	f_\#\left((\gamma_q^p)^\ell\right)=\gamma_q^s \,.
	$$
	By \cite[30.1]{0261.57009}, there are numbers $\varepsilon_i\in\{\pm 1\}$ and a permutation $\sigma\in S_m$ such that
	$$
	\ell \varepsilon_{\sigma(i)}p_{\sigma(i)}\equiv s_i \pmod q \quad 1\le i\le m .
	$$
	It remains to show that $\prod_i \varepsilon_i = \varepsilon$. Suppose not. By \cite[29.6]{0261.57009}, the isometry $\psi$ induced by the one given by \eqref{eq:explicit-isometry} is homotopic to $f$ through homotopy-equivalences. In particular, $f$ and $\psi$ have the same orientational behaviour, which is a contradiction.
	
	To prove $(2)\to (1)$, we note that $\Psi$ given by \eqref{eq:explicit-isometry} is an $\varepsilon$-isometry.
\end{proof}
\begin{remark}
	Note that for odd $q$, the lens spaces $\Le(q;p_1,\ldots,p_m)$ have a unique spin structure. In this case, (1) of Theorem~\ref{thm:lensspace-or-isometry} is equivalent to the statement that $\Le$ and $\Le'$ are $\varepsilon$-spin-isometric.
\end{remark}

\begin{theorem}\label{thm:lensspace-or-spin-isometry}
	Let $q$ be even and $\varepsilon\in\{\pm 1\}$. Let $\Le:=\Le(q;p_1,\ldots,p_m;h)$ and $\Le':=\Le(q;s_1,\ldots,s_m;h')$. Then the following assertions are equivalent:
	\begin{enumerate}
		\item $\Le$ and $\Le'$ are $\varepsilon$-spin-isometric.
		\item There is a number $\ell\in\Z$, a permutation $\sigma\in S_m$ and there are numbers $\varepsilon_i\in\{\pm 1\}$ such that
		\begin{eqnarray*}
			l p_{\sigma(i)} \varepsilon_{\sigma(i)} &\equiv & s_i \pmod {q} \quad   \forall 1\le i\le m\,,\\
			\prod_{i=1}^m \varepsilon_i & = & \varepsilon\,, \\
			h+h'+h_q^p+h_q^s &\equiv & \frac 1q \sum_{i=1}^{m} \left(l p_{\sigma(i)} \varepsilon_{\sigma(i)}-s_i\right) \pmod 2\,.
		\end{eqnarray*}
	\end{enumerate}
\end{theorem}
\begin{proof}
	The statement about the isometry and the orientation is clear from the previous theorem. Since the mapping of spin structures is a homotopy invariant of a map (in the class of homotopy equivalences), we can work with the isometry $\psi:\Le\to\Le'$ induced by the one given by  \eqref{eq:explicit-isometry}. Its lift $\Psi \in\Ot(2m)$ is covered by two elements $\pm\widetilde{\Psi}\in\Pin(2m)$. Now the relation \eqref{eq:generator-relation} lifts to $\Spin(2m)$ as
	$$
	\tau_h((\gamma_q^p)^\ell)=\widetilde{\Psi}^{-1}\cdot\tau'_{h'}(\gamma_q^s)\cdot\widetilde\Psi
	$$
	if and only if $h+h'+h_q^p+h_q^s \equiv \tfrac 1q \sum_{i=1}^{m} \left(l p_{\sigma(i)} \varepsilon_{\sigma(i)}-s_i\right) \pmod 2$.
	
\end{proof}

For any lens space $\Le(q;p_1,\ldots,p_m)$ we can always, either by one of the two last theorems or simply by choosing another generator $(\gamma_q^p)^k,\,k\in\Z$ with $(q,k)=1$, find a $(+1)$-spin-isometric lens space of the form $L(q;1,s_2,\ldots,s_m)$. In particular, every three dimensional lens space can be written as $\Le(q;1,p)$, which we will abbreviate as $\Le(q;p)$ from now on. Theorems~\ref{thm:lensspace-or-isometry} and \ref{thm:lensspace-or-spin-isometry} then take the following form.
\begin{corollary}\label{cor:3dim-or-isometry}
	Let $\varepsilon\in\{\pm 1\}$. The lens spaces $\Le(q;p)$ and $\Le(q;s)$ are $\varepsilon$-isometric if and only if
	$$
	p\equiv\varepsilon s\pmod q \quad \textrm{ or } \quad p\cdot s\equiv\varepsilon\pmod q \, .
	$$
\end{corollary}

\begin{corollary}\label{cor:3dim-or-spin-isometry}
	Let $q$ be even and $\varepsilon\in\{\pm 1\}$. The lens spaces $\Le(q;p;h)$ and $\Le'=\Le(q;s;h')$ are $\varepsilon$-spin-isometric if and only if
	\begin{equation*}
		\varepsilon p \equiv s\pmod {q} \quad  \textrm{and}  \quad h+h'+h_q^p+h_q^s  \equiv \frac{p-\varepsilon s}{q} \pmod 2
	\end{equation*}
	\center{or}
	\begin{equation*}
		\varepsilon p\cdot s\equiv 1 \pmod {q}\quad \textrm{and} \quad h+h'+h_q^p+h_q^s \equiv \frac{p s-\varepsilon}{q} \pmod 2\,.
	\end{equation*}
\end{corollary}
To facilitate the understanding of Corollaries~\ref{cor:3dim-or-isometry} and \ref{cor:3dim-or-spin-isometry}, we provide the following
\begin{example}
	 The lens space $\Le(7;2)$ is $+1$-spin-isometric to $\Le(7;4)$ because $2\cdot4\equiv 1\pmod 7$ whereas $\Le(17;4)$ is $-1$-spin-isometric to itself because $4^2\equiv -1\pmod {17}$. In particular, $\Le(17;4)$ has symmetric spectrum.	 
	 Analogously, $\Le(8;3;0)$ is $+1$-spin-isometric to $\Le(8;3;1)$ because $3^2\equiv 1\pmod 8$ and $0+1+0+0\equiv\tfrac{3\cdot3-1}{8}\pmod 2$ and $\Le(10;3;0)$ is $-1$-spin-isometric to $\Le(10;3;1)$. Note that there is a $-1$-spin-isometry of $\Le(q;p;h)$, $h\in\{0,1\}$, if and only if $p^2\equiv -1\pmod q$ and $\tfrac{p^2+1}{q}$ is even, which is never the case.
\end{example}
At last, we cite a theorem that characterises those lens space, which are Riemannian homogeneous.

\begin{theorem}[\cite{pre05830219}, Corollary 2.7.2]
	The lens space $\Le(q;p_1,\ldots,p_m)$ is homogeneous if and only if $p_i\equiv\pm p_j\pmod q$ for all $1\leq i<j\leq m$. In particular, two homogeneous lens spaces of the same dimension and volume are isometric.
\end{theorem}

\section{The Spectrum of the Sphere and its Quotients}\label{sec:spectra}

In this section, we describe the spectrum of the Dirac operator on the sphere and its quotients, the spherical space forms. We then specialise the formulas to lens spaces and end the section with a theorem about the spectral rigidity of homogeneous lens spaces.

The spectrum of the Dirac operator on the three dimensional sphere was first calculated in \cite{Hitchin}, though the round metric was only one among a large class of metrics for which Hitchin calculated the spectrum. In \cite{Sulanke1979}, S. Sulanke calculated the spectrum of $S^n$ for all $n\geq2$ by a representation theoretic approach. C. B\"ar found an alternative and shorter method to calculate the spectrum in \cite{0748.53022} (cf.\ \cite{0848.58046}) using Killing spinors, which also paved the way for the description of the spectrum on spherical space forms. It is this approach which we will follow and use in this and the following sections.

\begin{theorem}[\cite{Hitchin}, \cite{Sulanke1979}, \cite{0848.58046}]
	\label{thm:spherespectrum}
	The eigenvalues of the Dirac operator on the round sphere $S^n$ are
	$$
	\pm \left(\frac n2 +k\right),\; k\in\N_0
	$$
	with corresponding multiplicities
	$$
	\mult_{S^n}\left(\pm\left(\frac n2 +k\right)\right) = 2^{[\frac{n}{2}]}\binom{n+k-1}{k}\,.
	$$
\end{theorem}

We now pass to spherical space forms $\Gamma\bs S^{2m-1}$, where $\Gamma\subseteq\SO(2m)$ is a finite and freely acting group of orientation preserving isometries of $S^{2m-1}$. Restricting to odd dimensions is no loss of generality since the only spherical space forms in even dimensions are the sphere $S^{2m}$ itself and real projective space $\mathbb{PR}^{2m}$, which is not orientable and, in particular, not spin. Suppose that the quotient $\Gamma\bs S^{2m-1}$ is spin, and let a spin structure be given by $\tau:\Gamma\to\Spin(2m)$. The spinor fields on $\Gamma\bs S^{2m-1}$ can be identified with the $\Gamma$-invariant spinor fields on $S^{2m-1}$ by a unitary isomorphism (see \cite[Proposition 1.4.2]{1186.58020}). In particular, the eigenspinor fields of the Dirac operator on $\Gamma\bs S^{2m-1}$ can be identified with the $\Gamma$-invariant eigenspinor fields of the Dirac operator on $S^{2m-1}$. It follows that the eigenvalues of the Dirac operator on $\Gamma\bs S^{2m-1}$ are $\pm \left(\tfrac {2m-1}2+k \right)$, $k\ge 0$, with corresponding multiplicities 
$$
0\le \mult_{(\Gamma\bs S^{2m-1},\tau)}\left(\pm\left(\tfrac {2m-1}2 +k\right)\right) \le \mult_{ S^{2m-1}}\left(\pm\left(\tfrac {2m-1}2 +k\right)\right)\,.
$$
We weave these multiplicities into two power series.

\begin{definition}
	Let $\Gamma\bs S^{2m-1}$ be a spherical space form equipped with a spin structure $\tau:\Gamma\to\Spin(2m)$. The generating functions of (the spectrum of the Dirac operator on) $\Gamma\bs S^{2m-1}$ are
	$$
	F_\pm^{(\Gamma\bs S^{2m-1},\tau)}(z):=\sum_{k=0}^{\infty} \mult_{(\Gamma\bs S^{2m-1},\tau)}\left(\pm\left(\frac {2m-1}2+k\right)\right)z^k \,.
	$$
\end{definition}

Using the multiplicities on $S^{2m-1}$, a standard argument shows that these power series converge absolutely for $|z|<1$. In accordance with Definition~\ref{def:spin-isometry} we now make the
\begin{definition}
	Let $M$ and $N$ be compact Riemannian spin manifolds. Then $M$ and $N$ are \textit{$+1$-isospectral} if the spectra of their Dirac operators $D_M$ and $D_N$ coincide, where each eigenvalue is counted with its multiplicity. The Riemannian spin manifolds $M$ and $N$ are \textit{$-1$-isospectral} if the following condition is met: $\lambda\in\R$ is an eigenvalue of $D_M$ with multiplicity $m=\mult_M(\lambda)$ if and only if $-\lambda$ is an eigenvalue of $D_N$ with multiplicity $m=\mult_N(\lambda)$. The manifolds $M$ and $N$ are \textit{isospectral} if they are $\varepsilon$-isospectral for some $\varepsilon\in\{\pm1\}$.
\end{definition}

\begin{proposition}\label{prop:sphericalspaceforms-isospectrality-cond}
	Let $\Gamma\bs S^{2m-1}$ and $\Gamma'\bs S^{2m-1}$ be spherical space forms with spin structures $\tau:\Gamma\to\Spin(2m)$ and $\tau':\Gamma'\to\Spin(2m)$, respectively. Then $\Gamma\bs S^{2m-1}$ and $\Gamma'\bs S^{2m-1}$ are $(+1)$-isospectral if and only if $F_\pm^{(\Gamma\bs S^{2m-1},\tau)}=F_\pm^{(\Gamma'\bs S^{2m-1},\tau')}$ and $(-1)$-isospectral if and only if $F_\pm^{(\Gamma\bs S^{2m-1},\tau)}=F_\mp^{(\Gamma'\bs S^{2m-1},\tau')}$.
\end{proposition}

For the following Theorem, denote by $\chi^\pm:\Spin(2m)\to\C$ the positive and negative half-spin characters, respectively.

\begin{theorem}[{\cite[Theorem 2]{0848.58046}}]\label{thm:mero-ext-gen-func}
	Let $\Gamma\bs S^{2m-1}$ be a spherical space form equipped with a spin structure $\tau:\Gamma\to\Spin(2m)$. Then the eigenvalues of the Dirac operator are $\pm\left(\tfrac {2m-1}2+k\right),\, k\geq0$, with multiplicities determined by
	\begin{equation}\label{eq:mero-ext-gen-func}
		F_\pm^{(\Gamma\bs S^{2m-1},\tau)}(z)=\frac 1{|\Gamma|}\sum_{\gamma\in\Gamma}\frac{\chi^\mp(\tau(\gamma))-z\cdot\chi^\pm(\tau(\gamma))}{\det\left(\Id-z\cdot\gamma\right)}\,.
	\end{equation}
\end{theorem}

Denote by $\xi := \xi_n$ the $n$-th root of unity $e^{{2\pi i}/n}$. Then, for lens spaces, Theorem~\ref{thm:mero-ext-gen-func} takes the following form.
\begin{corollary}\label{cor:generating_functions_lens_spaces}
	Let $\Le=\Le(q;p_1,\ldots,p_m)$. If $q$ is odd, the generating functions of $\Le$ are
	\begin{align}
		\begin{split}\label{eqn:ls-q_odd-mero-cont_0}
			F_+^{(\Le,\tau)} (z)= \frac{2^{m-1}}{q}\sum_{k=0}^{q-1}\frac{\displaystyle\sum_{\epsilon_1\cdots\epsilon_m=(-1)^{m+1}}\xi_{2q}^{(q+1)k\sum_j\epsilon_j p_j} -z\cdot \!\!\!\!\!\!\!\sum_{\epsilon_1\cdots\epsilon_m=(-1)^{m}}\xi_{2q}^{(q+1)k\sum_j\epsilon_j p_j}}{\displaystyle\prod_{j=1}^m (\xi_q^{k p_j}-z)(\xi_q^{- k p_j}-z)}\,,
		\end{split}\\
		\begin{split}\label{eqn:ls-q_odd-mero-cont_1}
			F_-^{(\Le,\tau)} (z)= \frac{2^{m-1}}{q}\sum_{k=0}^{q-1}\frac{\displaystyle\sum_{\epsilon_1\cdots\epsilon_m=(-1)^{m}}\xi_{2q}^{(q+1)k\sum_j\epsilon_j p_j} -z\cdot \!\!\!\!\!\!\!\sum_{\epsilon_1\cdots\epsilon_m=(-1)^{m+1}}\xi_{2q}^{(q+1)k\sum_j\epsilon_j p_j}}{\displaystyle\prod_{j=1}^m (\xi_q^{k p_j}-z)(\xi_q^{- k p_j}-z)}\,.
		\end{split}
	\end{align}
	If $q$ and $m$ are even and $\Le$ is equipped with the spin structure $\tau_h$, then the generating functions of $\Le$ are
	\begin{align}
		\begin{split}\label{eqn:ls-q_even-mero-cont_0}
		F_+^{(\Le,\tau_h)} (z)=\frac{2^{m-1}}{q}\sum_{k=0}^{q-1}(-1)^{k(h+h_q^p)}\,\frac{\displaystyle\sum_{\epsilon_1\cdots\epsilon_m=(-1)^{m+1}}\xi_{2q}^{k \sum_j\epsilon_j p_j} -z\cdot \!\!\!\!\!\!\!\sum_{\epsilon_1\cdots\epsilon_m=(-1)^{m}}\xi_{2q}^{k \sum_j\epsilon_j p_j}}{\displaystyle\prod_{j=1}^m (\xi_q^{k p_j}-z)(\xi_q^{-k p_j}-z)}\,,
		\end{split}\\
		\begin{split}\label{eqn:ls-q_even-mero-cont_1}
			F_-^{(\Le,\tau_h)} (z)=\frac{2^{m-1}}{q}\sum_{k=0}^{q-1}(-1)^{k(h+h_q^p)}\,\frac{\displaystyle\sum_{\epsilon_1\cdots\epsilon_m=(-1)^{m}}\xi_{2q}^{k \sum_j\epsilon_j p_j} -z\cdot \!\!\!\!\!\!\!\sum_{\epsilon_1\cdots\epsilon_m=(-1)^{m+1}}\xi_{2q}^{k \sum_j\epsilon_j p_j}}{\displaystyle\prod_{j=1}^m (\xi_q^{k p_j}-z)(\xi_q^{-k p_j}-z)}\,.
		\end{split}
	\end{align}
\end{corollary}
\begin{proof}
	The values of the half-spin characters on the image of the homomorphisms inducing the spin structures can be found on \cite[p. 290]{MR781344}.
\end{proof}

\begin{remark}
	Heat kernel methods for the Dirac Laplacian $D^2$ on an arbitrary compact spin manifold show that the dimension as well as the volume of $M$ are spectrally determined (see \cite{1037.58015}). Thus, for a lens space $\Le(q;p_1,\ldots,p_m)$, $m$ and $q$ are spectrally determined. 
\end{remark}

The first application of the preceding corollary is the following spectral rigidity result for homogeneous lens spaces.

\begin{theorem}\label{thm:homo-lensspace-rigidity}
	Let $\Le=\Le(q;p_1,\ldots,p_m)$ and $\Le'=\Le(q;s_1,\ldots,s_m)$ be lens spaces with fixed spin structures. Let $\Le$ is homogeneous.
	If $\Le'$ is isospectral to $\Le$, then $\Le'$ is homogeneous as well and so, in particular, isometric to $\Le$. Moreover, if $m = 2$, $q > 2$, and $\Le'$ is $\varepsilon$-isospectral to $\Le$, then $\Le'$ and $\Le$ are $\varepsilon$-spin-isometric and carry the same spin structure.
\end{theorem}
\begin{proof}
	Let $F_\pm$ and $F'_\pm$ be the generating functions of $\Le$ and $\Le'$, respectively. Formulas~\eqref{eqn:ls-q_odd-mero-cont_0}-\eqref{eqn:ls-q_even-mero-cont_1} show that the poles of $F_\pm$ are precisely the $q$-th roots of unity and that these are at most of order $2m$. In fact, the term for $k=0$ generates a pole of order $2m-1$ at $z=1$ and in case $q$ is even, the term for $k=q/2$ generates a pole of order $2m-1$ at $z=-1$. By the homogeneity assumption, the denominator of every term for $k\not\in \{ 1,q/2\}$ has a zero of order $m$. In case $m$ is even, the numerators of these terms have real coefficients, hence, no zeros cancel the ones from the denominator. Let $m$ be odd and fix $1<k_0<q$. Since the coefficients of the numerator of the term for $k=k_0$ are complex conjugate to the coefficients of the numerator of the term for $k=q-k_0$, at least one of these terms has a pole of order $m$ at $z=\xi_q^k$. If $\Le$ and $\Le'$ are isospectral, then, by Proposition~\ref{prop:sphericalspaceforms-isospectrality-cond}, $F'_\pm$ must have a pole of order $m$ at every $q$-th root of unity that is not $1$ or $-1$, hence $s_i\equiv\pm s_j\pmod q$ for all $1\leq i<j\leq m$.
	
	Now let $m=2$ and $q>2$. If $q$ is odd, the multiplicities of the first positive and respectively negative eigenvalue of the Dirac operator on $\Le=\Le(q;1)$ are $F_+^{(L,\tau)}(0)=2$ and $F_-^{(L,\tau)}(0)=0$. If $q$ is even, then $F_+^{(L,\tau_0)}(0)=2$ and $F_-^{(L,\tau_0)}(0)=0$, whereas $F_+^{(L,\tau_1)}(0)=0=F_-^{(L,\tau_1)}(0)$. To distinguish the positive from the negative spectrum when $L$ is endowed with the spin structure $\tau_1$, we note that $\lim_{z\to-1}(1+z)^3F_+^{(L,\tau_1)}(z)=\frac 2q=-\lim_{z\to-1}(1+z)^3F_-^{(L,\tau_1)}(z)$.
\end{proof}

\begin{remark}\label{rmk:rp3-spectral-symmetry}
	Let $\Le=\Le(2;1)=\mathbb{RP}^3$, then $F_\pm^{(\Le,\tau_0)}=\frac{1}{(1-z)^3}\pm\frac{1}{(1+z)^3}=F_\mp^{(\Le,\tau_1)}$. This symmetry is generated by the isometry of $\mathbb{RP}^3$ corresponding to the choices $\ell=1,\,\varepsilon_1=-\varepsilon_2=1$ and $\sigma=(1 2)$ (see Theorem~\ref{thm:lensspace-or-spin-isometry}). Thus, by Theorem~\ref{thm:homo-lensspace-rigidity}, the spin manifold $\mathbb{RP}^3$ is the only three dimensional homogeneous lens space for which the spectrum is invariant under a simultaneous change of orientation and spin structure.
\end{remark}

\section{Three dimensional Lens Spaces}\label{sec:three-dim}

In this section we consider lens spaces of dimension three. We denote by $q$ a positive integer, by $p$ and $s$ integers that are coprime to $q$ and $p,s\not\equiv\pm 1\pmod q$. Furthermore, let $p^*,s^*\in\Z$ be any integers such that $p\cdot p^*\equiv s\cdot s^*\equiv 1\pmod q$.

We note that in the three dimensional case, the generating functions of $\Le(q;p)$ simplify to
$$
F_\pm(z)=\frac{2}{q}\sum_{k=0}^{q-1}\frac{\cos\frac{k(p\mp1)(q+1)}{q}\pi-z\cdot \cos\frac{k(p\pm1)(q+1)}{q}\pi}{\left(\xi_q^k-z\right)\left(\xi_q^{-k}-z\right)\left(\xi_q^{kp}-z\right)\left(\xi_q^{-kp}-z\right)}
$$
for odd $q$,  whereas the generating functions of $\Le(q;p;h)$ for even $q$ simplify to
$$
F_\pm(z)=\frac{2}{q}\sum_{k=0}^{q-1}(-1)^{k(h+h_q^p)}\frac{\cos\frac{k(p\mp1)}{q}\pi-z\cdot \cos\frac{k(p\pm1)}{q}\pi}{\left(\xi_q^k-z\right)\left(\xi_q^{-k}-z\right)\left(\xi_q^{kp}-z\right)\left(\xi_q^{-kp}-z\right)}\,.
$$

We do not give proofs for Lemma~\ref{lem:lem1}, Corollary~\ref{cor:cor0}, Lemma~\ref{lem:lem2}, and Corollary~\ref{cor:cor1} as the statements are the same as those of the corresponding lemmata and corollaries in \cite{0415.58018} and the proofs go through with at most minor modifications.

	\begin{lemma}\label{lem:lem1}
		\begin{align*}
			(q, p+1) = (q, p^\ast+1) & \\
			(q, p-1) = (q, p^\ast-1) &
		\end{align*}
	\end{lemma}
	
	\begin{corollary}\label{cor:cor0}
		Let $k$ be an integer such that
		\begin{align*}
			k(p\pm 1) \not\equiv 0 \pmod q \,.
		\end{align*}
		Then
		\begin{align*}
			k(p^\ast\pm 1) \not\equiv 0 \pmod q \,.
		\end{align*}
	\end{corollary}
	
	\begin{lemma}\label{lem:lem2}
		If $\Le(q;p)$ and $\Le(q;s)$ are isospectral, then
		\begin{align*}
			(q,p-1) = (q, s-1) & \\
			(q,p+1) = (q, s+1) &
		\end{align*}
		or
		\begin{align*}
			(q,p-1) = (q, s+1) & \\
			(q,p+1) = (q, s-1) & \,.
		\end{align*}
		In particular, $ ((q,p-1),(q,p+1))=
		\begin{cases}
		1 & \text{if } q \text{ is odd} \\
		2 & \text{if } q \text{ is even} \,.
		\end{cases} $
	\end{lemma}
	
	\begin{corollary}\label{cor:cor1}
		If $\Le(q;p)$ and $\Le(q;s)$ are isospectral and $k$ is an integer satisfying
		$$
		k(p\pm 1)\not\equiv 0 \pmod q \,,
		$$
		then
		$$
		k(s\pm 1)\not\equiv 0 \pmod q \,.
		$$
	\end{corollary}
	
	\begin{proposition}
		Let $k\in\Z$ such that $k(p\pm 1)\not\equiv 0\pmod q$. If $q$ is odd, the residues of the generating functions $F_\pm$ of the lens space $\Le(q;p)$ at $z=\xi_q^k$ are
		\begin{gather}\label{eq:residuumoddq}
			-\frac{2i}{q}\frac{\xi_q^k}{\left(1-\xi_q^{2k}\right)^{2}}\cdot\\
			\Bigg(\left(\cos\frac{k(p\mp1)(q+1)}{q}\pi-\xi_q^k \cos\frac{k(p\pm1)(q+1)}{q}\pi\right)\left(\cot\frac{k(p-1)}{q}\pi-\cot\frac{k(p+1)}{q}\pi\right)+\notag\\ \left(\cos\frac{k(p^{*}\mp1)(q+1)}{q}\pi-\xi_q^k \cos\frac{k(p^{*}\pm1)(q+1)}{q}\pi\right)
			\left(\cot\frac{k(p^{*}-1)}{q}\pi-\cot\frac{k(p^{*}+1)}{q}\pi\right)\Bigg)\notag\,.
		\end{gather}
		If $q$ is even, the residues of the generating functions $F_\pm$ of the lens space $\Le(q;p;h)$ at $z=\xi_q^k$ are
		\begin{gather}
			 -\frac{2i}{q}\frac{\xi_q^k}{\left(1-\xi_q^{2k}\right)^{2}}\cdot\label{eq:residuumevenq}\\
			 \Bigg((-1)^{k(h+h_q^p)}\left(\cos\frac{k(p\mp1)}{q}\pi-\xi_q^k \cos\frac{k(p\pm1)}{q}\pi\right)\left(\cot\frac{k(p-1)}{q}\pi-\cot\frac{k(p+1)}{q}\pi\right)+\notag \\
			 (-1)^{k(p^*(h+h_q^p)+u(p,p^*))}\left(\cos\frac{k(p^{*}\mp1)}{q}\pi-\xi_q^k \cos\frac{k(p^{*}\pm1)}{q}\pi\right)\left(\cot\frac{k(p^{*}-1)}{q}\pi-\cot\frac{k(p^{*}+1)}{q}\pi\right)\Bigg)\,,\notag
		\end{gather}
		where $u(p,p^*)$ is defined by $p\cdot p^*=u(p,p^*)q+1$.
	\end{proposition}
	\begin{proof}
		The condition on $k$ ensures that the $F_\pm$ have a pole of order one at $z=\xi_q^k$. There are precisely four terms that contribute to the residue of $F_\pm$ at $z=\xi_q^k$. Let $q$ be odd. Then we calculate straightforwardly:
		\begin{equation*}
			\begin{split}
				\lim\limits_{z\to\xi_q^k}(\xi_q^k-z)F_\pm & = 
				\frac 2q \lim\limits_{z\to\xi_q^k}(\xi_q^k-z)\sum_{l=0}^{q-1}\frac{\cos\frac{l(p\mp1)(q+1)}{q}\pi-z \cos\frac{l(p\pm1)(q+1)}{q}\pi}{\left(\xi_q^l-z\right)\left(\xi_q^{-l}-z\right)\left(\xi_q^{lp}-z\right)\left(\xi_q^{-lp}-z\right)}\\
				& = \frac 4q\left(\frac{\cos\frac{k(p\mp1)(q+1)}{q}\pi-\xi_q^k \cos\frac{k(p\pm1)(q+1)}{q}\pi}{\left(\xi_q^{-k}-\xi_q^k\right)\left(\xi_q^{kp}-\xi_q^k\right)\left(\xi_q^{-kp}-\xi_q^k\right)} + 
				\frac{\cos\frac{kp^*(p\mp1)(q+1)}{q}\pi-\xi_q^k \cos\frac{kp^*(p\pm1)(q+1)}{q}\pi}{\left(\xi_q^{kp^*}-\xi_q^k\right)\left(\xi_q^{-kp^*}-\xi_q^k\right)\left(\xi_q^{-k}-\xi_q^k\right)}\right)\\
				& = \frac 4q \frac{\xi_q^k}{1-\xi_q^{2k}}\left(\frac{\cos\frac{k(p\mp1)(q+1)}{q}\pi-\xi_q^k \cos\frac{k(p\pm1)(q+1)}{q}\pi}{\left(1-\xi_q^{k(1-p)}\right)\left(1-\xi_q^{k(1+p)}\right)}+\frac{\cos\frac{kp^*(p\mp1)(q+1)}{q}\pi-\xi_q^k \cos\frac{kp^*(p\pm1)(q+1)}{q}\pi}{\left(1-\xi_q^{k(1-p^*)}\right)\left(1-\xi_q^{k(1+p^*)}\right)}\right)\,.\\				
			\end{split}
		\end{equation*}
		Using the formula $\cot \theta = \tfrac {2i}{e^{2i\theta}-1}+i$, we transform the denominators to
		\begin{equation*}
			\begin{split}
				\frac{1}{\left(1-\xi_q^{k(1-p)}\right)\left(1-\xi_q^{k(1+p)}\right)} & = -\frac 1 {\xi_q^{k(p+1)}\left(1-\xi_q^{-k(p-1)}\right)\left(1-\xi_q^{-k(p+1)}\right)}\\
			& =  \frac 1 {\xi_q^{k(p+1)}\left(\xi_q^{-k(p-1)}-\xi_q^{-k(p+1)}\right)}\left(\frac{1}{1-\xi_q^{-k(p+1)}}-\frac{1}{1-\xi_q^{-k(p-1)}} \right) \\
			& = -\frac{1}{1-\xi_q^{2k}}\frac{i}{2}\left(\cot\frac{k(p-1)}{q}\pi -\cot\frac{k(p+1)}{q}\pi\right)\,.
			\end{split}
		\end{equation*}
		The case when $q$ is even is very similar.		
	\end{proof}
	
	\begin{definition}
		Let $k$ be an integer such that $k(p\pm 1)\not\equiv 0\pmod q$. If $q$ is odd, let
		\begin{equation*}
			I_q^{p;k} := \xi_q^{\frac{q+1}{2}k p}\left(\cot\frac{k(p-1)}{q}\pi - \cot\frac{k(p+1)}{q}\pi\right)\,.
		\end{equation*}
		If $q$ is even, define
		\begin{equation*}
			I_q^{p;k;h} := (-1)^{k(h+h_q^p)}\, \xi_{2q}^{k p}\left(\cot\frac{k(p-1)}{q}\pi - \cot\frac{k(p+1)}{q}\pi\right)
		\end{equation*}
		for $h\in\Z$.
	\end{definition}
	
	For the the next corollary, let $\sigma:\C\to\C$ denote complex conjugation. 
	\begin{corollary}\label{cor:isospec-lensspaces-eqs}
		If $q$ is odd and the lens spaces $\Le(q;p)$ and $\Le(q;s)$ are $\varepsilon$-isospectral, then
		\begin{equation}\label{eq:transcendental-isospec-q_odd}
					I_q^{p;k} + I_q^{p^*;k} = \sigma^{\tfrac 12(1-\varepsilon)}\left(I_q^{s;k}+I_q^{s^*;k}\right)
		\end{equation}
		for every $k\in\Z$ s.t. $k(p\pm 1)\not\equiv 0 \pmod q$. Now assume $q$ is even. If the lens spaces $\Le(q;p;h)$ and $\Le(q;s;h')$ are $\varepsilon$-isospectral, then
		\begin{equation}\label{eq:transcendental-isospec-q_even}
				I_q^{p;k;h}+(-1)^{k\left(u(p,p^*)+h_q^p+h_q^{p^*}\right)} I_q^{p^*;k;h} = \sigma^{\tfrac 12(1-\varepsilon)}\left(I_q^{s;k;h'}+(-1)^{k\left(u(s,s^*)+h_q^s+h_q^{s^*}\right)} I_q^{s^*;k;h'}\right)
		\end{equation}
		for every $k\in\Z$ s.t. $k(p\pm 1)\not\equiv 0 \pmod q$.
	\end{corollary}
	\begin{proof}
		Assume $q$ is odd and the lens spaces $\Le=\Le(q;p)$ and $\Le'=\Le(q;s)$ are $+1$-isospectral. By Proposition~\ref{prop:sphericalspaceforms-isospectrality-cond}, the generating functions of these lens spaces satisfy $F_\pm^\Le=F_\pm^{\Le'}$. In particular, the residues of their poles at every $z=\xi_q^k$ with $k$ such that $k(p\pm1)\not\equiv0\pmod q$ coincide. Dividing the residues \eqref{eq:residuumoddq} by their common factor
		% $ -\frac{2i}{q}\frac{\xi_q^k}{\left(1-\xi_q^{2k}\right)^{2}}$
		$-2i/q\cdot \xi_q^k/(1-\xi_q^{2k})^2$ 
		, considering the imaginary parts of the resulting equations and dividing again by the common factor $\sin\left(2\pi i k/q\right)$, we obtain the set of equations
		\begin{align*}
			\cos\tfrac{k(p\mp 1)(q+1)}{q}\pi \left(\cot\tfrac{k(p-1)}{q}\pi - \cot\tfrac{k(p+1)}{q}\pi\right) + 
			\cos\tfrac{k(p^*\mp 1)(q+1)}{q}\pi \left(\cot\tfrac{k(p^*-1)}{q}\pi - \cot\tfrac{k(p^*+1)}{q}\pi\right) \\
			= \cos\tfrac{k(s\mp 1)(q+1)}{q}\pi \left(\cot\tfrac{k(s-1)}{q}\pi - \cot\tfrac{k(s+1)}{q}\pi\right) + 
			\cos\tfrac{k(s^*\mp 1)(q+1)}{q}\pi \left(\cot\tfrac{k(s^*-1)}{q}\pi - \cot\tfrac{k(s^*+1)}{q}\pi\right)\,.
		\end{align*}
		Addition and subtraction of the equations "+" and "-" yield, up to a common factor, the real and imaginary parts of equation \eqref{eq:transcendental-isospec-q_odd}. The case of $-1$-isospectrality is very similar. The case of even $q$ is similar, too.
	\end{proof}
	\begin{remark}
		One can show by either straightforward calculations or using Corollary~\ref{cor:3dim-or-spin-isometry} that the term $ (-1)^{u(p,p^*)+h_q^p+h_q^{p^*}}$ in the equations~\eqref{eq:transcendental-isospec-q_even} is an invariant of the $+1$-spin-isometry class of $\Le(q;p;h)\cong\Le(q;p^*;h+h_q^p+h_q^{p^*}+u(p,p^*))$.
	\end{remark}

	We denote by $\Q_q$ the $q$-th cyclotomic field $\Q(\xi_q)$. Its ring of integers is $\Z[\xi_q]$. Let $\lambda:=1-\xi_q\in\Q_q$. Then $(\lambda)$ is a prime ideal in $\Z[\xi_q]$ and we have $(q)=(\lambda)^{q-1}$. From now on, we consider on $\Q_q$ the $\lambda$-adic valuation and the corresponding metric. This provides a notion of convergence (see e.g.\ \cite[Chapter 17.1.2]{1082.11065}).
	
	\begin{lemma}\label{lem:maintechnical}
		Let $q\ge5$ be prime. Then for all $ 1\le l,k\le q-1 $, the number
		$$
		\xi_q^l\frac{\lambda}{1-\xi_q^k}\in\Z\left[\xi_q\right]
		$$
		has the convergent power series expansion
		\begin{align*}
			\xi_q^l \frac{\lambda}{1-\xi_q^k} = \frac 1k \Big( 1 & + \frac{-1-2l+k}{2}\lambda \\
											  & + \frac{-1+k^2-6kl+6l^2}{12} \lambda^2 \\					
											  & + \frac{-1+k^2-6kl-2k^2l+6l^2+6kl^2-4l^3}{24}\lambda^3 + \ldots\,,
		\end{align*}
		all of whose coefficients are elements of $\Z_{(q)}$, the localization of $\Z$ by the maximal ideal $(q)$.
	\end{lemma}
	\begin{proof}
	To prove the first claim, choose $1\le k^*\le q-1$ such that $kk^*\equiv 1 \pmod q$. Then
	\begin{align*}
		\frac{\lambda}{1-\xi_q^k}	& = \frac{1-\xi_q}{1-\xi_q^k} = \frac{1-\left(\xi_q^k\right)^{k^*}}{1-\xi_q^k}\\
									&=1+\xi_q^k+\xi_q^{2k}+\ldots +\xi_q^{k(k^*-1)}\,.
	\end{align*}
	To prove the statement about the power series, we expand the individual parts separately. The first part is
	\begin{align*}
		\xi_q^l  & = \left(1-\lambda\right)^l = \sum_{j=0}^{l}\binom lj (-\lambda)^j \\
		& = 1-l\lambda+\frac{l(l-1)}{2}\lambda^2-\frac{l(l-1)(l-2)}{6}\lambda^3+\ldots \,.
	\end{align*}
	Next, we expand $1-\xi_q^k$ into the power series
	\begin{align*}
		1-\xi_q^k & = 1-\left(1-\lambda\right)^k=\lambda k \sum_{j=0}^{k-1}\binom{k-1}j\frac{(-\lambda)^j}{j+1}\\
				  & = \lambda k \Big( 1-\frac{k-1}{2}\lambda + \frac{(k-1)(k-2)}{6}\lambda^2-\frac{(k-1)(k-2)(k-3)}{24}\lambda^3+\ldots \,.
	\end{align*}
	Because of $1\le k\le q-1$, all coefficients of this power series are elements of $\Z_{(q)}$. The same is then true for the series of $\lambda/(1-\xi_q^k)$ and its Cauchy product with the series of $1-\xi_q^l$. This is sufficient for the claimed convergence statement.
	\end{proof}
	
	\begin{theorem}
		\label{thm:mainthm}
		Let $q\in\N$ be prime. If two lens spaces with fundamental groups of order $q$ are $\varepsilon$-isospectral, then they are $\varepsilon$-spin-isometric.

	\end{theorem}
	\begin{proof}
		The case $q=2$ was discussed in Remark~\ref{rmk:rp3-spectral-symmetry}. For $q=3,5,7$ there are at most two different isometry classes, one of which is homogeneous. Thus, these are spectrally determined by Theorem~\ref{thm:homo-lensspace-rigidity}. Let $q\ge11$ and $\Le=\Le(q;p)$ and $\Le'=\Le(q;s)$ be two $+1$-isospectral lens spaces. By the same theorem we can assume $p\not\equiv\pm 1\pmod q$ and thus by assumption $s\not\equiv\pm 1\pmod q$.	
		By Corollary~\ref{cor:isospec-lensspaces-eqs} we have
		$$
			I_q^{p;1}+I_q^{p^*;1}=I_q^{p;1}+I_q^{p^*;1}\,.
		$$
		Using the formula $\cot \theta=\tfrac{2i}{1-e^{2i\theta}}-i$ and Lemma~\ref{lem:maintechnical}, we see that
		$\tfrac i2\lambda (I_q^{p;1}+I_q^{p^*;1})=\tfrac i2\lambda(I_q^{s;1}+I_q^{s^*;1})\in\Z[\xi_q] $. Let $ R=\{0,1,\ldots,q-1\} $ and
		\begin{align*}
			\begin{split}
			\frac i2\lambda(I_q^{p;1}+I_q^{p^*;1}) & = \sum_{n=0}^{\infty}g_n\lambda^n\,,\\
			\frac i2\lambda(I_q^{s;1}+I_q^{s^*;1}) & = \sum_{n=0}^{\infty}g'_n\lambda^n\\
			\end{split}
		\end{align*}		
		be the unique power series with $g_n,g'_n\in R$ for all $n\in\N$ (see e.g.\ \cite[Chapter 2.2 O]{0957.12005}). By assumption, $g_n=g'_n$ for all $n\in\N$. Furthermore, by Lemma~\ref{lem:maintechnical}
		\begin{equation*}
			\begin{split}
				g_{3} & \equiv \frac{1}{24}\sum_{(l,k)}\left(-\underbrace{\frac{1}{k}}_{=\mathtt{I}}+\underbrace{k}_{=\mathtt{II}}-\underbrace{6l}_{=\mathtt{III}}-\underbrace{2kl}_{=\mathtt{IV}}+\underbrace{\frac{6l^{2}}{k}}_{=\mathtt{V}}+\underbrace{6l^{2}}_{=\mathtt{VI}}-\underbrace{\frac{4l^{3}}{k}}_{=\mathtt{VII}}\right)\\
				&-\frac{1}{24}\sum_{(l',k')}\left(-\underbrace{\frac{1}{k'}}_{=\mathtt{I'}}+\underbrace{k'}_{=\mathtt{II'}}-\underbrace{6l'}_{=\mathtt{III'}}-\underbrace{2k'l'}_{=\mathtt{IV'}}+\underbrace{\frac{6l'^{2}}{k'}}_{=\mathtt{V'}}+\underbrace{6l'^{2}}_{=\mathtt{VI'}}-\underbrace{\frac{4l'^{3}}{k'}}_{=\mathtt{VII'}}\right)\pmod	q\,,
			\end{split}
		\end{equation*}
		where we sum over $(l,k)\in\left\{\left(\tfrac{q+1}{2}p,p-1\right),\allowbreak \left(\tfrac{q+1}{2}p^*,p^*-1\right)\right\}$ and $(l',k')\in\left\{\left(\tfrac{q+1}{2}p,p+1\right),\break \left(\tfrac{q+1}{2}p^*,p^*+1\right)\right\}$.
		
		We calculate the summands in pairs. Obviously, $\mathtt{III}-\mathtt{III'}\equiv\mathtt{VI}-\mathtt{VI'}\equiv 0 \pmod q$. Furthermore, we have
\begin{equation*}
\begin{split}
	\mathtt{I}-\mathtt{I}' & \equiv -\frac{1}{p-1}-\frac{1}{p^{*}-1}+\frac{1}{p+1}+\frac{1}{p^{*}+1}\\
	& \equiv -\frac{1}{p-1}-\frac{p}{1-p}+\frac{1}{p+1}+\frac{p}{1+p}\equiv2\pmod q\,,\\
\end{split}
\end{equation*}
\begin{equation*}
	\begin{split}
	\mathtt{II}-\mathtt{II}' & \equiv p-1+p^{*}-1-p-1-p^{*}-1\equiv-4\pmod q\,,\\
\end{split}
\end{equation*}
\begin{equation*}
	\begin{split}
	\mathtt{IV}-\mathtt{IV}' & \equiv -2\frac{q+1}{2}p(p-1)-2\frac{q+1}{2}p^{*}(p^{*}-1)\\
	& + 2\frac{q+1}{2}p(p+1)+2\frac{q+1}{2}p^{*}(p^{*}+1)\\
	& \equiv (q+1)(-p^{2}+p-\left(p^{*}\right)^{2}+p^{*}\\
	& + p^{2}+p+\left(p^{*}\right)^{2}+p^{*}\equiv2(q+1)(p+p^{*})\pmod q\,,\\
\end{split}
\end{equation*}
\begin{equation*}
	\begin{split}
	\mathtt{V}-\mathtt{V}' & \equiv \frac{3}{2}\left(q+1\right)^{2}\left(\frac{p^{2}}{p-1}+\frac{\left(p^{*}\right)^{2}}{p^{*}-1}-\frac{p^{2}}{p+1}-\frac{\left(p^{*}\right)^{2}}{p^{*}+1}\right)\\
	& \equiv \frac{3}{2}\left(q+1\right)^{2}\left(\frac{p^{2}}{p-1}-\frac{p^{*}}{p-1}-\frac{p^{2}}{p+1}-\frac{p^{*}}{p+1}\right)\\
	& \equiv \frac{3}{2}\left(q+1\right)^{2}\left(\frac{p^{2}-p^{*}}{p-1}-\frac{p^{2}+p^{*}}{p+1}\right)\\
	& \equiv \frac{3}{2}\left(q+1\right)^{2}\left(\frac{p^{3}+p^{2}-1-p^{*}-p^{3}-1+p^{2}+p^{*}}{p^{2}-1}\right)\\
	& \equiv \frac{3}{2}\left(q+1\right)^{2}\left(\frac{2p^{2}-2}{p^{2}-1}\right)\equiv3(q+1)^{2}\pmod q\,,\\
\end{split}
\end{equation*}
\begin{equation*}
	\begin{split}
	\mathtt{VII}-\mathtt{VII}' & = -\frac{\left(q+1\right)^{3}}{2}\left(\frac{p^{3}}{p-1}+\frac{\left(p^{*}\right)^{3}}{p^{*}-1}-\frac{p^{3}}{p+1}-\frac{\left(p^{*}\right)^{3}}{p^{*}+1}\right)\\
	& \equiv -\frac{\left(q+1\right)^{3}}{2}\left(\frac{p^{3}-\left(p^{*}\right)^{2}}{p-1}-\frac{p^{3}+\left(p^{*}\right)^{2}}{p+1}\right)\\
	& \equiv -\frac{\left(q+1\right)^{3}}{2}\left(\frac{p^{4}-p^{*}+p^{3}-\left(p^{*}\right)^{2}-p^{4}-p^{*}+p^{3}+\left(p^{*}\right)^{2}}{p^{2}-1}\right)\\
	& \equiv -\frac{\left(q+1\right)^{3}}{2}\left(\frac{2p^{3}-2p^{*}}{p^{2}-1}\right)\\
	& \equiv -\left(q+1\right)^{3}\left(\frac{p^{2}-\left(p^{*}\right)^{2}}{p-p^{*}}\right)\equiv-\left(q+1\right)^{3}\left(p+p^{*}\right)\pmod q\,.
\end{split}
\end{equation*}
Putting everything together, we obtain
\begin{equation*}
	\begin{split}
	g_{3} & \equiv \frac{1}{24}(2-4+2(q+1)\left(p+p^{*}\right)\\
	& + 3(q+1)^{2}-\left(q+1\right)^{3}\left(p+p^{*}\right))\\
	& \equiv \frac{1}{24}\left(-2+3(q+1)^{2}-\left(p+p^{*}\right)(q+1)\left(q^{2}+2q-1\right)\right)\pmod q\,.
	\end{split}
\end{equation*}
Analogously, 
\begin{equation*}
	\begin{split}
		g'_{3} & \equiv \frac{1}{24}\left(-2+3(q+1)^{2}-\left(s+s^{*}\right)(q+1)\left(q^{2}+2q-1\right)\right)\pmod q\,.
	\end{split}
\end{equation*}
Since, by assumption, $g_3 = g'_3$, we have
\begin{equation}\label{eqn:final1}
p+p^*\equiv s+s^* \pmod q\,.
\end{equation}
Squaring both sides of \eqref{eqn:final1} yields
$$
p^2+2+(p^*)^2\equiv s^2+2+(s^*)^2\pmod q
$$
which, after subtracting $4$ on both sides, leads to
$$
(p-p^*)^2\equiv (s-s^*)^2\pmod q\,.
$$
This means 
\begin{equation}\label{eqn:final2}
(p-p^*)\equiv \pm (s-s^*)\pmod q\,.
\end{equation}
Addition and subtraction of equations \eqref{eqn:final1} and \eqref{eqn:final2} yields
$$
p\equiv s \pmod q\text{	or	} p\equiv s^*\pmod q\,,
$$
which means, by Corollary~\ref{cor:3dim-or-isometry}, that $\Le$ and $\Le'$ are $+1$-isometric. Since for odd $q$ the lens spaces $\Le(q;p)$ admit only one spin structure, $\Le$ and $\Le'$ are trivially $+1$-spin-isometric.
\end{proof}

\begin{remark}
	A nontrivial question is whether the assumption that $q$ be prime in Theorem~\ref{thm:mainthm} is necessary for the conclusion to hold. High precision computer calculations of the numbers $I_q^{p;k}$ and $I_q^{p;k;h}$ for $2\le q\le10^4$, $p$ from a full set of representatives of $\Z_q$, $1\le k<q, k(p\pm1)\not\equiv 0\pmod q$ and all $h\in\{0,1\}$ suggest that the condition can probably be dropped. These calculations were verified with a computer and the exact method from \cite{2014arXiv1412.2599B} for all parameters in the same range, so that the following conjecture is reasonable:
\end{remark}
\begin{conjecture}\label{con:all_q}
	Any two three dimensional $\varepsilon$-isospectral lens spaces are $\varepsilon$-spin-isometric.
\end{conjecture}
\begin{remark}\label{rem:main-conjecture_consequences}
Note that Theorem~\ref{thm:mainthm} implies for each odd prime $q$ that $\Le(q;p)$ has symmetric spectrum if and only if $p^2\equiv-1\pmod q$. If Conjecture~\ref{con:all_q} is true, then this will hold for all odd $q$. Furthermore, under this assumption, $\Le(q;p;0)$ will have the same spectrum as $\Le(q;p;1)$ if and only if $p^2\equiv1\pmod q$ and $(p^2-1)/q$ is odd. Expanding on Remark~\ref{rmk:rp3-spectral-symmetry}, $\Le(q;p;0)$ will then be $-1$-isometric to $\Le(q;p;1)$ if and only if $p^2\equiv-1\pmod q$ and $(p^2+1)/q$ is odd.
\end{remark}

\section{the $\eta$-invariant}\label{sec:eta-invariant}

The $\eta$-invariant measures the asymmetry of the Dirac spectrum in dimension $n\equiv 3\pmod 4$ and appears in the Atiyah-Patodi-Singer index  Theorem (cf.\ \cite[Chapter 8.7]{1186.58020}). We here present a formula due to C. B\"ar for the $\eta$-invariant on spherical space forms and, more specifically, on lens spaces. Using this formula and the method from the last section, we prove that, in case $q$ is a prime, any two lens spaces $\Le=\Le(q;p)$ and $\Le'=\Le(q;s)$ with $\eta^\Le=\varepsilon\eta^{\Le'}$ are $\varepsilon$-spin-isometric. Example~\ref{exm:eta-invariant} shows that this statement does not generalize to all $q\in\N$. Note that (modulo the $\varepsilon$-statement) both of these results are already known (see \cite{MR926246}), though the method of proof of the first one is new.

\begin{theorem}[{\cite[Theorem 5.2	]{0995.58020}}]
	Let $\Gamma\setminus S^{(2m-1)}$ be a spherical space form with spin structure given by $\tau:\Gamma\to\Spin(2m)$. The $\eta$-invariant of $\Gamma\setminus S^{2m-1}$ is
	$$
 \eta = \frac{2}{|\Gamma|}\sum_{\gamma\in\Gamma\setminus\{\Id_{2m}\}}\frac{(\chi^--\chi^+)(\tau(\gamma))}{\det(\Id_{2m}-\gamma)} \,.
 $$
\end{theorem}

\begin{corollary}
	Let $\Le=\Le(q;p_1,\ldots,p_m)$ be a lens space. If $q$ is odd, then the $\eta$-invariant of the Dirac operator on the lens space $L$ is given by
	$$
		\eta^{\Le} = (-1)^{m/2+1}\frac{1}{q2^{m-1}}\sum_{k=1}^{q-1}\prod_{j=1}^m \csc\frac{(q+1)kp_j}{q}\pi \,.
	$$
	If $q$ and $m$ are even and $\Le$ is equipped with the spin structure $\tau_h$, then the $eta$-invariant of $\Le$ is
	$$
		\eta^{\Le} = (-1)^{m/2+1}\frac{1}{q2^{m-1}}\sum_{k=1}^{q-1}(-1)^{k(h+h_q^{p*})}\prod_{j=1}^m \csc\frac{kp_j}{q}\pi \,.
	$$
\end{corollary}

\begin{theorem}\label{thm:etaeqisom}
	Let $q\in\N$ be a prime, $\varepsilon\in\{\pm 1\}$, and let $\Le = \Le(q;p)$, $\Le' = \Le(q;p)$ be lens spaces such that
	$\eta^{\Le} = \varepsilon \eta^{\Le'}$. Then $\Le$ and $\Le'$ are $\varepsilon$-spin-isometric.
\end{theorem}
\begin{proof}
	The case $q=2$ was already discussed in Remark~\ref{rmk:rp3-spectral-symmetry}. Let $q\ge3$. 
	Represent $\csc x$ as $2i e^{-i x}\frac{1}{1-e^{-2 i x}}$. A simple generalization of Lemma~\ref{lem:maintechnical} to the case of products of the form $\xi_q^l\frac{\lambda}{1-\xi_q^k}\frac{\lambda}{1-\xi_q^n}$ yields convergent power series in $\lambda$ for 
	$\eta^{\Le}$ and $\eta^{\Le'}$, along with expressions for the coefficients of $\lambda^0$. Comparing the latter yields $p\equiv \varepsilon s \pmod q$ or $p\equiv \varepsilon s^\ast \pmod q$.
\end{proof}

\begin{example}\label{exm:eta-invariant}
	There are six isometry classes of three dimensional lens spaces whose fundamental groups have order 25. Representatives are $\Le(25;1)$, $\Le(25;2)$, $\Le(25;3)$, $\Le(25;4)$, $\Le(25;7)$ and $\Le(25;9)$. The corresponding $\eta$-invariants are $\frac{52}{25}$, $-\frac{2}{5}$, $\frac{4}{5}$, -$\frac{2}{25}$, $0$ and $-\frac{2}{25}$. In particular, Theorem \ref{thm:etaeqisom} does not generalize to all $q\in\N$.
\end{example}

\bibliographystyle{halpha}
\bibliography{library}

\end{document}